\date{}
\documentclass[11pt]
{article}

\usepackage{amsmath,amsfonts,amssymb,amsthm}
\usepackage{latexsym, xspace, enumerate}
\usepackage[mathscr]{eucal}
\usepackage[all]{xypic}
\usepackage{hyperref}
\usepackage{amscd}
\usepackage{vmargin}

\setmarginsrb{20mm}{10mm}{20mm}{25mm}%
          {10mm}{7mm}{10mm}{10mm}

\newtheorem{theorem}{Theorem}[section]

\newtheorem{proposition}[theorem]{Proposition}
\newtheorem{claim}[theorem]{Claim}
\newtheorem{lemma}[theorem]{Lemma}
\newtheorem{fact}[theorem]{Fact}
\newtheorem{corollary}[theorem]{Corollary}

\theoremstyle{definition}

\newtheorem{example}[theorem]{Example}
\newtheorem{remark}[theorem]{Remark}

\newcommand{\N}{\mathbb N}

\newcommand{\Q}{\mathbb Q}
\newcommand{\Z}{\mathbb Z}
\newcommand{\Prm}{\mathbb P}
\newcommand{\T}{\mathbb{T}}

\def\nub{\mathrm{nub}}

\def\J{\mathbb{J}}
\newcommand{\NB}{}

\makeatletter
\def\Ddots{\mathinner{\mkern1mu\raise\p@
\vbox{\kern7\p@\hbox{.}}\mkern2mu
\raise4\p@\hbox{.}\mkern2mu\raise7\p@\hbox{.}\mkern1mu}}
\makeatother

\numberwithin{equation}{section}

\title{Scale function vs Topological entropy}

\author{Federico Berlai\footnote{The first named author was partially supported by the European Research Council (ERC) grant of Prof. Goulnara Arzhantseva, grant agreement no. 259527.} \and Dikran Dikranjan$^\dag$ \and Anna Giordano Bruno\footnote{The second and the third named authors were partially supported by ``Progetti di Eccellenza 2011/12'' of Fondazione CARIPARO.}}

\begin{document}

\maketitle

\begin{abstract}
In the realm of topological automorphisms of totally disconnected locally compact groups, the scale function introduced by Willis in \cite{Willis} is compared with the topological entropy. We prove that the logarithm of the scale function is always dominated by the topological entropy and we provide examples showing that this inequality can be strict.  Moreover, we give a condition equivalent to the equality between these two invariants.
Various properties of the scale function, inspired by those of the topological entropy, are presented.
\end{abstract}

\bigskip
\noindent 
\textrm{\small{Key words: scale function, topological entropy, totally disconnected locally compact group, automorphism}}

\section{Introduction}

The scale function for  inner automorphisms of totally disconnected locally compact groups was introduced by Willis \cite{Willis} and developed in his later works; among them we mention \cite{Willis5,Willis2,Willis3}, where the scale function was defined for topological automorphisms of such groups as well. A wealth of results concerning the explicit computation of the scale function in $p$-adic Lie groups and in linear groups over local fields were obtained by Gl\" ockner  \cite{Gl1,Gl2,GW}.

On the other hand, Adler, Konheim and McAndrew introduced in \cite{AKM} the topological entropy for continuous selfmaps of compact spaces, while later on Bowen in \cite{B} gave a different definition of topological entropy for uniformly continuous selfmaps of metric spaces, and this was extended to uniformly continuous selfmaps of uniform spaces by Hood in \cite{hood}. As explained in detail in \cite{DG-islam}, this definition can be significantly simplified in the case of continuous endomorphisms of totally disconnected locally compact groups.

For a topological automorphism of a totally disconnected locally compact group, the scale function and the topological entropy seem to be strongly related. A question in this direction was posed by Thomas Weigel, who asked for a possible relation of the scale function with either the topological entropy or the algebraic entropy. Even if they do not coincide in general, we see in this paper that the values of the scale function and of the topological entropy can be obtained in a similar way, and this permits to find the precise relation between them.
Further aspects of the connection of the scale function to the topological and the algebraic entropy are discussed in the forthcoming paper \cite{BDGT}. 

\medskip
So in this paper we are mainly concerned with a totally disconnected locally compact group $G$ and a topological automorphism $\phi:G\to G$; when not explicitly said, we are assuming to be under these hypotheses.
It is worth recalling immediately that a totally disconnected locally compact group $G$ has as a local base at $e_G$ the family $\mathcal B(G)$ of all open compact subgroups of $G$, as proved by van Dantzig in \cite{vD}.

\bigskip
\bigskip
We start now giving the precise definition of scale function as it was introduced in \cite{Willis3}.
For $G$ a totally disconnected locally compact group and $\phi:G\to G$ a topological automorphism, the \emph{scale} of $\phi$ is 
\begin{equation}\label{(S)}
s_G(\phi) = \min\{s_G(\phi,U): U \in \mathcal B(G)\},
\end{equation}
where $$s_G(\phi,U)=[\phi(U):U\cap\phi(U)].$$ 
Note that for every $U\in\mathcal B(G)$ the index $s_G(\phi,U)$ is finite as $U\cap \phi(U)$ is open and $\phi(U)$ is compact, so the value $s_G(\phi)$ of the scale function of $\phi$ is always a natural number. We use the notations $s(\phi,U)$ and $s(\phi)$ when the group $G$ is clear from the context.

It is worth to observe immediately that the scale function of any topological automorphism $\phi$ of $G$ is trivial (i.e., $s(\phi)=1$), whenever the  group $G$ is either compact or discrete.

\medskip
Since the scale function is defined as a minimum, there exists $U \in \mathcal B(G)$ for which this minimum realizes, that is $s(\phi) =s(\phi,U)$, and such $U$ is called \emph{minimizing} for $\phi$ in \cite{Willis2}. Let $\mathcal M(G,\phi)$ be the subfamily of $\mathcal B(G)$ consisting of all compact open subgroups of $G$ that are minimizing for $\phi$, that is
$$\mathcal M(G,\phi)=\{U\in\mathcal B(G):U\ \text{minimizing for}\ \phi\}.$$
Moreover, we say that a subgroup $U$ of $G$ is $\phi$-\emph{invariant} if $\phi(U)\subseteq U$, \emph{inversely} $\phi$-\emph{invariant} if $U\subseteq \phi(U)$ (i.e., $\phi^{-1}(U)\subseteq U$), and $\phi$-\emph{stable} if $\phi(U)=U$ (i.e., $U$ is both $\phi$-invariant and inversely $\phi$-invariant).

\smallskip
It is easy to see that $\mathcal M(G,\phi)$ contains all $\phi$-invariant or inversely $\phi$-invariant $U\in\mathcal B(G)$; in particular, 
$s(\phi)=1$ precisely when there exists a $\phi$-invariant $U\in\mathcal B(G)$, that is, $\mathcal M(G,\phi)=\{U\in\mathcal B(G): U\ \text{$\phi$-invariant}\}$ when $s(\phi)=1$ (see Lemma \ref{invariantsubgroup}). 

\smallskip
On the other hand, if one has to use only the definition of scale function, the subgroups minimizing for $\phi$ that are not $\phi$-invariant or inversely $\phi$-invariant become quite hard to come by, since in \eqref{(S)} one has to check all subgroups from the large filter base $\mathcal B(G)$. So, in order to characterize and find minimizing subgroups, a different approach is adopted by Willis and we describe it in what follows. For $U\in \mathcal B(G)$ let 
\begin{equation}\label{upiu}
U_{\phi+}= \bigcap_{n\in\N} \phi^n(U)\ \quad \text{and}\ \quad U_{\phi-}= \bigcap_{n\in\N} \phi^{-n}(U);
\end{equation}
and also
\begin{equation}\label{upiupiu}
U_{\phi++}=\bigcup_{n\in\N} \phi^n(U_{\phi+})\ \quad\text{and}\ \quad U_{\phi--}=\bigcup_{n\in\N} \phi^{-n}(U_{\phi-}).
\end{equation}
Note that $U_{\phi-}=U_{\phi^{-1}+}$ is $\phi$-invariant and $U_{\phi+}=U_{\phi^{-1}-}$
 is inversely $\phi$-invariant (for further properties of these subgroups see Lemma \ref{invariantsubgroup} and the diagram \eqref{carpet}).
When the automorphism $\phi$ is clear from the context it is omitted from these notations. 

The subgroup $U$ is said to be:
\begin{itemize}
\item[(a)] \emph{tidy above} for $\phi$ if $U = U_+U_-$ (or equivalently, $U=U_-U_+$); 
\item[(b)] \emph{tidy below} for $\phi$  if $U_{++}$ is closed;
\item[(c)] \emph{tidy} for $\phi$ if it is tidy above and tidy below for $\phi$.
\end{itemize}

The consequence of the so-called ``tidying procedure" given in \cite{Willis2} is the following fundamental theorem showing that the minimizing subgroups are precisely the tidy subgroups, namely $$\mathcal M (G,\phi)=\{U\in\mathcal B(G):U\ \text{tidy for}\ \phi\}.$$

\begin{theorem}\label{TP}\emph{\cite[Theorem 3.1]{Willis2}}
Let $G$ be a totally disconnected locally compact group, $\phi:G\to G$ a topological automorphism and $U\in\mathcal B(G)$. Then $U$ is minimizing for $\phi$ if and only if $U$ is tidy for $\phi$. In this case $$s(\phi)=[\phi(U_+):U_+].$$
\end{theorem}

Note that the index $[\phi(U_+):U_+]$ is finite as $U_+ = U \cap \phi(U_+)$ and $U$ is open, so $U_+$ is open in $\phi(U_+)$, while $\phi(U_+)$ is compact. 

\bigskip
We pass now to recall the definition of topological entropy in this setting, 
following \cite{DG-islam}. Let $G$ be a totally disconnected locally compact group, $\phi:G\to G$ a continuous endomorphism and $U\in\mathcal B(G)$. For an integer $n\geq 0$ let 
\begin{equation}\label{uenne}
U_n=\bigcap_{k=0}^n\phi^k(U)\quad\text{and}\quad U_{-n}=\bigcap_{k=-n}^0\phi^k(U).
\end{equation}
The \emph{topological entropy of $\phi$ with respect to $U$} is given by the following limit, which is proved to exist,
$$H_{top}(\phi,U)=\lim_{n\to \infty}\frac{\log[U:U_{-n}]}{n};$$
and the \emph{topological entropy} of $\phi$ is $$h_{top}(\phi)=\sup\{H_{top}(\phi,U):U\in\mathcal B(G)\}.$$

As mentioned above the scale function of any topological automorphism of any totally disconnected compact group is trivial. This is not the case for the topological entropy; indeed, for example for a prime $p$ the topological entropy of the left Bernoulli shift of $\Z(p)^\Z$ is $\log p$ (see Example \ref{beta}). So it is immediately clear that the topological entropy does not always coincide with the logarithm of the scale function.

\medskip
The following limit free formula for the computation of the topological entropy was proved in \cite{DG-htop} and gives the possibility to easily compare the scale function with the topological entropy. An analogous formula for the topological entropy of continuous endomorphisms of totally disconnected compact groups was previously given in \cite{DG-limitfree}.

\begin{theorem}\label{lf}\emph{\cite{DG-htop}}
Let $G$ be a totally disconnected locally compact group, $\phi:G\to G$ a topological automorphism and $U\in\mathcal B(G)$; then 
$$H_{top}(\phi,U)=\log[\phi(U_+):U_+].$$
\end{theorem}

In \cite{DG-htop} this formula is applied to verify the basic properties of the topological entropy, well-known for compact groups (see \cite{St}), also for topological automorphisms of totally disconnected locally compact groups. These properties are the so-called Logarithmic Law, Invariance under conjugation, Monotonicity for subgroups and quotients, Weak Addition Theorem and Continuity for inverse limits (see Fact \ref{properties} below).

\paragraph*{Contents of the paper.}
\bigskip
\bigskip
The paper is organized as follows. 

\medskip
In the first part of Section \ref{scale-sec} we recall some basic properties of the tidy subgroups, which are applied in the following sections to prove the main results of this paper. Then we give a background on a subgroup considered in \cite{BW} and studied more in deed and given a name in \cite{Willis4}\NB, which is strongly related to tidy subgroups and the scale function. Namely, for $G$ a totally disconnected locally compact group and $\phi:G\to G$ a topological automorphism, the \emph{nub} of $\phi$ is the intersection of all subgroups $U\in\mathcal B(G)$ tidy for $\phi$, that is, in view of Theorem \ref{TP},
\begin{equation}\label{DefNUB}
\nub(\phi)=\bigcap\{U\in\mathcal{B}(G): U\in\mathcal M(G,\phi)\}.
\end{equation}
Clearly, $\mathcal M(G,\phi)$ is a local base at $e_G$ precisely when $\nub(\phi)=\{e_G\}$ (see Corollary \ref{tidylocalbase}). Moreover, $\nub(\phi)$ is $\phi$-stable and compact (see Fact \ref{nub_erg}), so $s(\phi\restriction_{\nub(\phi)})=1$.

\medskip
In Section \ref{comp} we compare the values of the logarithm of the scale function with those of the topological entropy. Indeed, we see that for $G$ a totally disconnected locally compact group and $\phi:G\to G$ a topological automorphism, they are respectively the $\min$ and the $\sup$ of the same subset of $\log\N_+$, that is
$$\{\log [\phi(U_+):U_+]:U\in\mathcal B(G)\}.$$
An immediate consequence is that the inequality 
\begin{equation}\label{logs<htop}
\log s(\phi)\leq h_{top}(\phi)
\end{equation}
holds in general (see Theorem \ref{logs<htop-th}). Moreover, this inequality can be strict, even in the non-compact (abelian) case (see Example \ref{example}), and this occurs precisely when the nub is not trivial. Indeed, the following theorem gives the precise relation between the scale function and the topological entropy, so answers the above mentioned question motivating this paper.

\begin{theorem}\label{nubbanale}
Let $G$ be a totally disconnected locally compact group and $\phi:G\to G$ a topological automorphism. Then $\log s(\phi)=h_{top}(\phi)$ if and only if $\nub(\phi)=\{e_G\}$.
\end{theorem}

In the first version of this paper Theorem \ref{nubbanale} was formulated as a conjecture. More precisely, the sufficiency of the condition $\nub(\phi)=\{e_G\}$ for the equality $\log s(\phi)=h_{top}(\phi)$ was proved (see Proposition \ref{nubb}), while we only conjectured the necessity of that condition. 
Two different proofs of the necessity, included at the end of Section \ref{comp}, were offered to us by Udo Baumgartner (a more topological one) and by Pablo Spiga (a more algebraically oriented one). 

\medskip
In Section \ref{ep} we give the properties of the scale function with respect to the typical properties of the topological entropy. Logarithmic Law, Invariance under conjugation and Monotonicity for subgroups and quotients were proved in \cite{Willis,Willis2}, while we see that also the Weak Addition Theorem holds true and we discuss Continuity for direct and inverse limits.
 
\medskip
In Section \ref{yf-bt} we present an explicit computation of the scale function of any topological automorphism of $\Q_p^n$, where $\Q_p$ denotes the field of $p$-adic numbers and $n$ is a positive integer (see Theorem \ref{formulaQpN}). This result is inspired by the so-called $p$-adic Yuzvinski Formula for the topological entropy. Indeed, the Yuzvinski Formula was proved in \cite{Y} by Yuzvinski and it gives the values of the topological entropy of topological automorphisms $\phi$ of $\widehat \Q^n$ in terms of the Mahler measure of the characteristic polynomial of $\phi$. A different and clear proof of the Yuzvinski Formula is given in \cite{LW}, and it is based on the computation of the topological entropy of topological automorphisms of $\Q_p^n$, that is the result we referred to above as $p$-adic Yuzvinski Formula.
A particular case of Theorem \ref{formulaQpN} (when all eigenvalues of $\phi$ belong to $\Q_p$) was mentioned without proof in \cite{Willis3}. Moreover, one can obtain this result directly from the general method given by Gl\"ockner in \cite{Gl2} for computing the scale function on $p$-adic Lie groups.

\smallskip
In the second part of Section \ref{yf-bt} we assume the totally disconnected locally compact group $G$ to be abelian. Under the hypothesis that $G$ is covered by its compact subgroups, that ensures total disconnectedness of the Pontryagin dual $\widehat G$ of $G$, we prove that 
$$s(\phi)=s(\widehat \phi),$$ 
where $\widehat\phi:\widehat G\to \widehat G$ is the dual automorphism of the topological automorphism $\phi$ of $G$ (see Theorem \ref{BTS}). This is a so-called Bridge Theorem, inspired by the analogous one from \cite{DG-btlca} connecting the topological entropy with the algebraic entropy in the same setting.

\subsection*{Notation and terminology}
As usual, $\N$ denotes the set of natural numbers and $\N_+$ the set of positive integers, $\Prm$ denotes the set of all prime numbers,
$\Z$ denotes the group of integers and $\T$ denotes the circle group with its usual topology.
For $p\in\Prm$, $\Z(p)=\Z/p\Z$ denotes the cyclic group of order $p$, $\J_p$ denotes the group/ring of $p$-adic integers and 
 $\Q_p$ denotes the field of $p$-adic numbers.
 
If $G$ and $H$ are topological groups we indicate by $G\cong H$ that they are topological isomorphic, that is, they are isomorphic both as groups and as topological spaces.
 
Let $F$ be a group and consider $G=F^{\Z}$; then the left Bernoulli shift on $G$ is the automorphism $\sigma\colon G\to G$ defined by $\sigma\bigl((f_i)_{i\in\Z}\bigr)=(f_{i+1})_{i\in\Z}$; if $F$ is a topological group and $G$ is endowed with the product topology, then $\sigma:G\to G$ is a topological automorphism.
 
Let $G$ be a topological abelian group, then the Pontryagin dual $\widehat{G}$ of $G$ is the (abelian) group of all continuous 
homomorphisms $\chi\colon G\to\T$ (i.e., characters), endowed with the compact-open topology. If $\phi\colon G\to G$ is a continuous endomorphism, then its 
dual homomorphism $\widehat{\phi}\colon \widehat{G}\to\widehat{G}$ is defined by $\widehat{\phi}(\chi)=\chi\circ\phi$ for every $\chi\in\widehat{G}$. If $G$ is a locally compact abelian group, so is its dual group $\widehat{G}$, and the dual endomorphism $\widehat \phi:\widehat G\to \widehat G$ is continuous. Moreover, if $G$ is finite then $G\cong \widehat{G}$, and $G$ is discrete if and only if $\widehat{G}$ is compact.
Recall also that, if $X\subseteq G$, the annihilator of $X$ in $\widehat G$ is
$X^\bot=\{\chi\in\widehat{G}:\chi(x)=0\ \forall x\in X\}$ and, if $Y\subseteq \widehat{G}$, the annihilator of $Y$ in $G$ is $Y^\bot=\{g\in G: \chi(g)=0\ \forall \chi\in Y\}.$

\subsection*{Acknowledgements}

It is a pleasure to thank Thomas Weigel for asking the question that inspired this work, George Willis for sending us his preprint \cite{Willis4} which was fundamental for give an answer to the question, Udo Baumgartner and Pablo Spiga for their kind permission to include here their proofs of our conjecture.
Last, but not least, we thank the referee for the sharp and useful comments and suggestions.

\section{Scale function and tidy subgroups}\label{scale-sec}

In the first part of this section we are mainly concerned with basic properties of tidy subgroups.

\medskip
The next lemma collects in particular known immediate examples of minimizing subgroups.

\begin{lemma}\label{invariantsubgroup}
Let $G$ be a totally disconnected locally compact group, $\phi\colon G\to G$ a topological automorphism and $U\in\mathcal B(G)$. Then:
\begin{itemize}
\item[(a)] if $U$ is $\phi$-invariant then $U$ is minimizing for $\phi$ and $s(\phi)=s(\phi,U)=1$;
\item[(b)] $U$ is minimizing for $\phi$ if and only if $U$ is minimizing for $\phi^{-1}$, i.e., $\mathcal M(G,\phi)=\mathcal M(G,\phi^{-1})$;
\item[(c)] if $U$ is inversely $\phi$-invariant then $U$ is minimizing for $\phi$ and $s(\phi)=s(\phi,U)\geq 1$, with equality exactly when $U$ is also $\phi$-stable;
\item[(d)] consequently, the following conditions are equivalent:
\begin{itemize}
\item[(i)] $s(\phi) = 1$ 
\item[(ii)] there exists a $\phi$-invariant $U \in \mathcal B(G)$;
\item[(iii)] $\mathcal M(G,\phi)=\{U\in\mathcal B(G): U\ \text{$\phi$-invariant}\}$.
\end{itemize}
\end{itemize}
\end{lemma}
\begin{proof}
(a) follows immediately from the definition of scale function, (b) is \cite[Corollary 1]{Willis}, (c) follows from (a) and (b), while (d) from (a) and the definition of scale function.
\end{proof}

If $G$ is a totally disconnected locally compact group and $\phi:G\to G$ a topological automorphism, 
for a subgroup $U\in\mathcal B(G)$ one can consider the following diagram, helping to memorize better 
the subgroups defined in \eqref{upiu}, \eqref{upiupiu} and \eqref{uenne}, and their interrelations. Note that $U_+U_-$ need not be a subgroup of $G$;
this condition is satisfied exactly when $U$ is tidy above for $\phi$, and in this case the diagram is contained in the lattice of all subgroups of $G$.

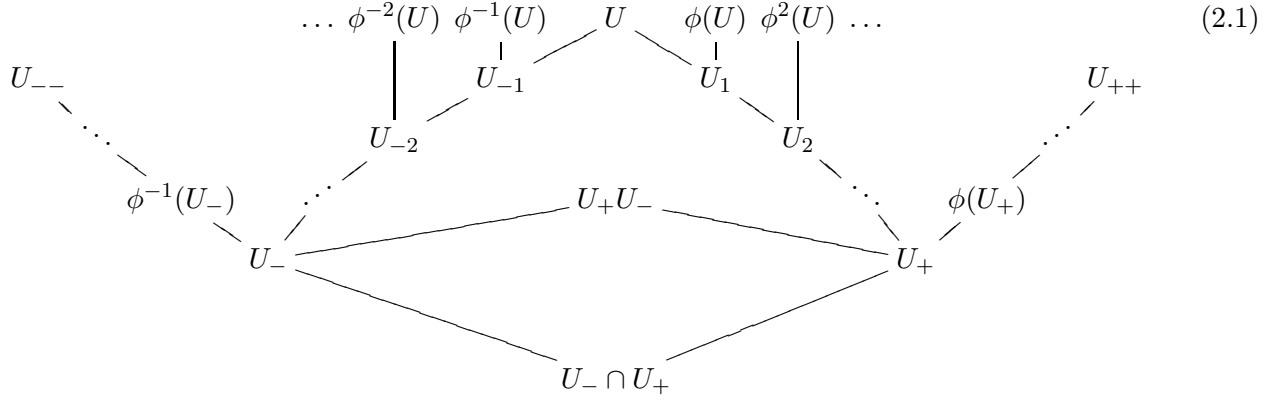
\begin{equation}\label{carpet}
\xymatrix@!R@-2.1pc{
&&& & \ldots &\phi^{-2}(U)\ar@{-}[dd] & \phi^{-1}(U)\ar@{-}[d] & U\ar@{-}[dl]\ar@{-}[dr] & \phi(U) \ar@{-}[d]& \phi^2(U) \ar@{-}[dd]& \ldots & &&&& \\
U_{--}&&&&& & U_{-1}\ar@{-}[dl] &  & U_1\ar@{-}[dr] &  & &&&& U_{++} \\
&\ddots\ar@{-}[ul]&&& & U_{-2}\ar@{-}[dl] && && U_2\ar@{-}[dr] &&&& \Ddots\ar@{-}[ur]&\\
&&\phi^{-1}(U_-)\ar@{-}[ul] & &\Ddots\ar@{-}[dl] &&& U_+U_-  &&& \ddots\ar@{-}[dr]&&\phi(U_+)\ar@{-}[ur]&&\\
&&&U_{-}\ar@{-}[ul]\ar@{-}[ddrrrr]\ar@{-}[urrrr] &&&&&&&& U_+\ar@{-}[ur] \ar@{-}[ddllll]\ar@{-}[ullll]&&&\\
&&&&&&& &&&&&&&\\
&&&&&&& U_-\cap U_+ &&&&&&&
}
\end{equation}
The motivation to introduce these subgroups is to measure the extent to which the subgroup $U$ is $\phi$-invariant or inversely $\phi$-invariant. Indeed, $U$ is $ \phi$-invariant if and only if $U_- = U$, while $U$ is inversely $\phi$-invariant if and only if 
$U_+= U$.

\medskip
The subgroup $U_+$ is  compact and it is the largest inversely $\phi$-invariant subgroup contained in $U$; moreover, we have an increasing 
chain of subgroups
$$U_+\subseteq\phi(U_+)\subseteq\ldots\subseteq\phi^{n}(U_+)\subseteq\ldots\subseteq U_{++}=\bigcup_{n\in\N}\phi^n(U_+),$$
where all indices $[\phi^{n+1}(U_+):\phi^{n}(U_+)]$ coincide with $[\phi(U_+):U_+]$ (so are finite, as noted above). Hence 
$U_{++}$, which is the increasing union of this chain, is a subgroup of $G$ that contains $U_+$. 
If $U$ is tidy below for $\phi$, that is $U_{++}$ is closed in $G$, then $U_{++}$ is locally compact with the subspace topology, hence a Baire space; so there exists an integer $n\geq0$ such that $\phi^n(U_+)$ 
is open in $U_{++}$; in this case $U_+$ is also open in $U_{++}$. The converse implication holds true as well, so we have the following lemma characterizing the subgroups tidy below for $\phi$.

\begin{lemma}
Let $G$ be a totally disconnected locally compact group, $\phi\colon G\to G$ a topological automorphism and $U\in\mathcal{B}(G)$. Then the following conditions are equivalent:
\begin{itemize}
\item[(a)]$U$ is tidy below for $\phi$;
\item[(b)]$U_{++}$ is locally compact;
\item[(c)]$U_+$ is open in $U_{++}$.
\end{itemize}
\end{lemma}

\NB The next Lemma \ref{taeq} was inspired by the proof of \cite[Lemma 1]{Willis} concerning the case of inner automorphisms; it provides characterizations of the tidy above subgroups. The following elementary fact from group theory is needed.

\begin{claim}\label{claimbase}
Let $G$ be a group and let $A$, $B$, $C$ be subgroups of $G$. If $C\subseteq B$ and $B\subseteq A\cdot C$, then $B=(A\cap B)C=C(A\cap B)$.
\end{claim}

\begin{lemma}\label{taeq}
Let $G$ be a totally disconnected locally compact group, $\phi\colon G\to G$ a topological automorphism and $U\in\mathcal{B}(G)$. Then the following conditions are equivalent:
\begin{itemize}
  \item[(a)] $U$ is tidy above for $\phi$;
  \item[(b)] $\phi (U)=\phi (U_+)(U\cap \phi (U))$;
  \item[(c)] $\phi^n (U)=\phi^n (U_+)U_n$ for every integer $n\geq 0$;
  \item[(d)] $U_+\cap uU_-\neq \emptyset$ for every $u\in U$. 
\end{itemize}
\end{lemma}
\begin{proof} 
(a)$\Rightarrow$(b) Let $U\in\mathcal{B}(G)$ be tidy above for $\phi$. This means that $U=U_-U_+$ and so $\phi(U)=\phi(U_-)\phi(U_+)$. Moreover, $\phi(U_-)\subseteq U_-\subseteq U$ and then $\phi(U)\subseteq U\phi(U_+)$. Now Claim \ref{claimbase} applied to $U$, $\phi(U)$ and $\phi(U_+)$ yields $\phi (U)=\phi (U_+)(U\cap \phi (U))$.

(b)$\Rightarrow$(c) Let $n\geq 0$. The inclusion $U_n\subseteq U$ is always satisfied, so
$$\phi(U_n)\subseteq \phi (U)=\phi (U_+)(U\cap \phi (U))\subseteq \phi (U_+)U,$$
thus Claim \ref{claimbase} applied to $U$, $\phi(U_n)$ and $\phi(U_+)$ yields 
\begin{equation}
\phi(U_n)=\phi (U_+)(U\cap \phi (U_n))=\phi (U_+)U_{n+1}.\label{equationUn}
\end{equation}
Using \eqref{equationUn} we prove by induction the condition in (c).
Indeed, the case $n=0$ is clear and the case $n=1$ is exactly the condition in (b). Now assume that $\phi^n (U)=\phi^n (U_+)U_n$. Therefore $\phi^{n+1}(U)=\phi^{n+1}(U_+)\phi(U_n)=\phi^{n+1}(U_+) U_{n+1}$, where the last equality follows from \eqref{equationUn} noting also that $\phi(U_+)\subseteq \phi^{n+1}(U_+)$.

(c)$\Rightarrow$(d) Let $u\in U$ and consider, for every integer $n\geq 0$, the subset $C_n(u)=U_+\cap u U_{-n}$. These subsets are compact and satisfy $C_{n+1}(u)\subseteq C_n(u)$. Moreover, since $\phi^n(U_{-n})=U_n$,
\begin{align*}
C_n(u)&=\{z\in U_+: z\in u U_{-n}\}\\
&=\{z\in U_+: u^{-1}\in z^{-1}U_{-n}\}\\
&=\{z\in U_+: \phi^n(u^{-1})\in \phi^n(z^{-1})U_n\}.
\end{align*}
Then $C_n(u)$ is non-empty in view of the condition in (c). By the compactness of $U_+$, the intersection $C=\bigcap_{n\in\N}C_n(u)$ is non-empty. Moreover, it coincides with $U_+\cap u U_-$; in fact, the inclusion $U_+\cap u U_-\subseteq C$ is clear. To verify the converse inclusion let $z\in C$, that exists since $C$ is non-empty; then $z\in U_+\cap u U_{-n}$ for every $n\geq 0$, in particular $z\in U_+$ and $u^{-1}z\in U_{-n}$ for every $n\geq 0$, that is $z\in U_+\cap uU_-$.
 
(d)$\Rightarrow$(a) For every $u\in U$ there exist $u_+\in U_+$ and $u_-\in U_-$ such that $u_+=uu_-$, that is $u=u_+(u_-)^{-1}$. This means that $U\subseteq U_+U_-$, that is $U$ is tidy above for $\phi$.
\end{proof}

This lemma has important consequences. In particular, the following corollary of Lemma \ref{taeq} and Theorem \ref{TP}, \NB which is contained in Step 1 of the proof of \cite[Theorem 3.1]{Willis2}, is one of the two main ingredients to prove in the next section the inequality announced in \eqref{logs<htop}.

\begin{corollary}\label{s=min}
Let $G$ be a totally disconnected locally compact group, $\phi\colon G\to G$ a topological automorphism and $U\in\mathcal{B}(G)$. Then $s(\phi,U)\geq [\phi (U_+):U_+]$; equality holds exactly when $U$ is tidy above for $\phi$.

In particular, $$s(\phi)=\min\{[\phi(U_+):U_+]:U\in\mathcal B(G)\}.$$ 
\end{corollary}
\begin{proof}
Since $\phi(U)\supseteq \phi(U_+)U_1$, where $U_1=U\cap \phi(U)$, we have 
$$s(\phi,U)=[\phi(U):U_1]\geq [\phi(U_+)U_1:U_1].$$ 
Moreover, 
$$[\phi(U_+)U_1:U_1]=[\phi(U_+):U_1\cap\phi(U_+)]=[\phi(U_+):U_+].$$ 
This proves that $s(\phi,U)\geq[\phi(U_+):U_+]$.

If $U$ is tidy above for $\phi$, then $\phi(U)=\phi(U_+)U_1$ by Lemma \ref{taeq}, and hence we have the equality $s(\phi,U)= [\phi (U_+):U_+]$.

From what we have just proved it follows that $s(\phi)\leq \min\{[\phi(U_+):U_+]:U\in\mathcal B(G)\}$. Equality holds, since Theorem \ref{TP} yields that $s(\phi)=s(\phi,V)=[\phi(V_+):V_+]$ for some $V\in\mathcal M(G,\phi)$.
\end{proof}

Another consequence of Lemma \ref{taeq} is the following result. \NB It was proved in \cite[Lemma 1]{Willis} in the case of inner automorphisms.

\begin{corollary}\label{esiste n}
Let $G$ be a totally disconnected locally compact group, $\phi\colon G\to G$ a topological automorphism and $U\in\mathcal{B}(G)$. 
There exists an integer $n\geq 0$ such that $U_n$ is tidy above for $\phi$. In particular, the subgroups tidy above for $\phi$ form a local base at $e_G$.
\end{corollary}
\begin{proof}
Consider the subfamily $\{\phi(U_n)\}_{n\geq 0}$ of $\mathcal B(G)$, and note that $\phi(U_n)\supseteq \phi(U_{n+1})$ for every $n\geq0$, and $\phi(U_+)=\bigcap_{n\in\N}\phi(U_n)$. Consider the set $\phi(U_+)U$, which is a compact and open neighborhood of $\phi(U_+)$. There exists an integer $n\geq 0$ such that $\phi(U_n)\subseteq \phi(U_+)U$. Apply now Claim \ref{claimbase} to $U$, $\phi(U_n)$ and $\phi(U_+)$ to obtain 
$$\phi(U_n)=\phi(U_+)(U\cap \phi(U_n))=\phi(U_+)U_{n+1}.$$
Since $U_+=(U_n)_+$ and $U_{n+1}=U_n\cap \phi(U_n)$ we have that
$$\phi(U_n)=\phi((U_n)_+)(U_n\cap \phi(U_n)).$$
In view of Lemma \ref{taeq} this means that $U_n$ is tidy above for $\phi$.
\end{proof}

\bigskip
In the second part of this section we recall the properties of the nub (that is the intersection of all tidy subgroups, see \eqref{DefNUB}), starting with the following useful characterization of this remarkable subgroup.

\begin{fact}\label{nub_erg}\emph{\cite[Corollary 4.7]{Willis4}}
Let $G$ be a totally disconnected locally compact group and $\phi:G\to G$ a topological automorphism.
Then $\nub(\phi)$ is the largest $\phi$-stable compact subgroup of $G$ having no proper $\phi$-stable relatively open subgroups.
\end{fact}

This fact implies that when the nub is finite, then it is trivial. Indeed, if $\nub(\phi)$ is finite then $\{e_G\}$ is open in $\nub(\phi)$, consequently $\{e_G\}$ is a $\phi$-stable relatively open subgroup of $\nub(\phi)$, and thus $\nub(\phi)=\{e_G\}$ by Fact \ref{nub_erg}. 

Moreover, it is worth to observe that always $s(\phi\restriction_{\nub(\phi)})=1$ as $\nub(\phi)$ is compact.

\medskip
We know that the family $\mathcal{B}(G)$ of compact open subgroups of $G$ is a local base at $e_G$ and that every $U\in\mathcal{B}(G)$ contains a compact open subgroup that is tidy above for $\phi$ by Corollary \ref{esiste n}. Moreover, \cite[Corollary 4.3]{Willis4} asserts that a subgroup $U\in\mathcal B(G)$ is tidy below for $\phi$ if and only if $\nub(\phi)\subseteq U$ (see also \cite[Lemma 3.31]{BW}\NB). So we have the following result, where (b) can be deduced from (a) via Lemma \ref{invariantsubgroup}, and \NB (a) is essentially contained in \cite[Theorem 3.32]{BW}, where several conditions equivalent to $\nub(\phi)=\{e_G\}$ are given. 

\begin{corollary}\label{tidylocalbase}
Let $G$ be a totally disconnected locally compact group and $\phi:G\to G$ a topological automorphism. Then: 
\begin{itemize}
\item[(a)] $\mathcal M(G,\phi)$ is a local base at $e_G$ if and only if $\nub(\phi)=\{e_G\}$;
\item[(b)] if $s(\phi)=1$, then $\nub(\phi) = \{e_G\}$ if and only if $G$ has a local base at $e_G$ consisting of $\phi$-invariant compact open subgroups.
\end{itemize} 
\end{corollary}

If $s(\phi)=1$, we have $U\in\mathcal M(G,\phi)$ precisely when $U$ is $\phi$-invariant by Lemma \ref{invariantsubgroup}, so in this case
\begin{equation}
\nub(\phi)=\bigcap\{U\in\mathcal{B}(G): U\text{ is }\phi\text{-invariant}\}. \label{dagdag}
\end{equation}
Now \eqref{dagdag} allows us to extend the definition of the nub also to arbitrary continuous endomorphisms of totally disconnected compact groups (note that $[\phi(U):U\cap \phi(U)]$ is finite for every $U\in\mathcal B(G)$). 

\medskip
Let us see some examples of computation of the nub.

\begin{example}\label{June13}
\begin{itemize}
    \item[(a)] Let $G$ be a totally disconnected locally compact group and $\phi:G\to G$ a topological automorphism. If $\phi$ is periodic (i.e., $\phi^m=id_G$ for some integer $m>0$), then $G$ has a base of $\phi$-invariant compact open subgroups, so $\nub(\phi) = \{e_G\}$ as noted after Corollary \ref{tidylocalbase}. 
   \item[(b)] If $G= \prod_p N_p$, where $p$ is a prime and each $N_p$ is a finitely generated $\J_p$-module, then $G$ has a base of fully invariant compact open subgroups (namely, $\{mG:m\in \N_+\}$), so $\nub(\phi) = \{e_G\}$ for every continuous endomorphism of $G$.
   \item[(c)] Let $G = F^\Z$, where $F$ is an arbitrary finite group. Then $\nub(\sigma)= G$, where $\sigma: G \to G$ is the left Bernoulli shift, (see Fact \ref{nub_erg}).
   \item[(d)] Let $G$ be a totally disconnected compact (i.e., profinite) abelian group. Then for every continuous endomorphism
$\phi: G \to G$ one can completely describe $\nub(\phi)$ by using the dual endomorphism $\widehat{\phi}: \widehat{G} \to \widehat{G}$.
Indeed, $$\nub(\phi) = t_{\widehat{\phi}}(\widehat{G})^\bot,$$ where $ t_{\widehat{\phi}}(\widehat{G})$ is the sum of all finite 
$\widehat{\phi}$-invariant subgroups of the discrete torsion abelian group $\widehat{G}$; in terms of \cite{DG}, $t_{\widehat{\phi}}(\widehat{G})$ is the Pinsker subgroup of $\widehat{\phi}$, defined as the largest $\widehat{\phi}$-invariant subgroup of $\widehat{G}$ where the restriction of $\widehat{\phi}$ has algebraic entropy zero.

According to \cite{DG}, $\nub(\phi) $ is the largest $\phi$-invariant closed subgroup of $G$ where the restriction of $\phi$ acts ergodically, or, equivalently, has strongly positive topological entropy; this means that the induced endomorphism $\overline \phi : G/\nub(\phi)\to G/ \nub(\phi)$ is the Pinsker factor of $\phi$, that is $h_{top}(\overline \phi)=0$ and this is the largest factor with this property (see \cite{DG} for more details).

\item[(e)] As noted in \cite{Willis4}, the dynamical property of the subgroup $ \nub(\phi)$  from item (d) remains true in the non-abelian case too. Namely, $\phi$ acts transitively on $ \nub(\phi)$, and $\nub(\phi)$ is the largest closed $\phi$-invariant subgroup of $G$ where $\phi$ acts ergodically.

\item[(f)] For any integer $n>0$ and every topological automorphism $\phi:\Q_p^n\to \Q_p^n$, $\nub(\phi)$ is trivial. Indeed, being a compact subgroup of $\Q_p^n$, $\nub(\phi) \cong \J_p^m$ for some $0 \leq m \leq n$. By (b) we can conclude that $\nub(\phi)$ has plenty of proper $\phi$-stable open subgroups. According to Fact \ref{nub_erg}, this implies $ m=0$.
\end{itemize}
\end{example}

\section{The scale function and the topological entropy}\label{comp}

It follows from Corollary \ref{s=min} that
$$\log s(\phi)=\min\{\log [\phi(U_+):U_+]:U\in\mathcal B(G)\}.$$
Furthermore, Theorem \ref{lf} yields
$$h_{top}(\phi)=\sup\{\log[\phi(U_+):U_+]:U\in\mathcal B(G)\}.$$
This gives the inequality announced in \eqref{logs<htop}:

\begin{theorem}\label{logs<htop-th}
Let $G$ be a totally disconnected locally compact group and $\phi:G\to G$ a topological automorphism. Then 
\begin{equation}\label{geq}
\log s(\phi)\leq h_{top}(\phi).
\end{equation}
\end{theorem}

We observe immediately that if $G$ is compact, then the inequality in Theorem \ref{logs<htop-th} can be strict in a trivial way. 
Indeed, $G$ compact implies $\log s(\phi)=0$, while the topological entropy $h_{top}(\phi)$ can be positive, as in the next example.

\begin{example}\label{beta}
For a prime $p$ let $G=\Z(p)^\Z$ and $\sigma:G\to G$ the left Bernoulli shift. Then $h_{top}(\sigma)=\log p$ (see \cite{AKM,St}); on the other hand, we have seen that $s(\sigma)=1$ and $\nub(\sigma)=G$ in Example \ref{June13}(c).
\end{example}

The inequality in Theorem \ref{logs<htop-th} can be obtained also in different way based on an equivalent definition of the scale function, as explained in the next remark.

\begin{remark}
For $G$ a totally disconnected locally compact group and $\phi:G\to G$ a topological automorphism, it was proved in \cite[Theorem 7.7]{M} that, for any $U\in\mathcal B(G)$,
$$\log s(\phi)=\lim_{n\to\infty}\frac{\log[\phi^n(U):U\cap \phi^n(U)]}{n}.$$
This gives immediately that $\log s(\phi)\leq h_{top}(\phi)$, because $[\phi^n(U):U\cap\phi^n(U)]=[U:\phi^{-n}(U)\cap U]$ as $\phi$ is an automorphism, and $[U:\phi^{-n}(U)\cap U]\leq [U:U_{-n}]$ as $U_{-n}\subseteq\phi^{-n}(U)\cap U$.
\end{remark}

Let $G$ be a totally disconnected locally compact group and $\phi\colon G\to G$ a continuous endomorphism.
Since $H_{top}(\phi,-)$ is antimonotone, that is,
\begin{center}
if $U,V\in\mathcal B(G)$ and $U\subseteq V$, then $H_{top}(\phi,V)\leq H_{top}(\phi,U)$, 
\end{center}
by the definition, it is clear that to compute the topological entropy $h_{top}(\phi)$ it suffices to take the supremum of $H_{top}(\phi,U)$ when $U$ ranges in a local base at $e_G$ of $G$:

\begin{claim}\label{basesuff}
Let $G$ be a totally disconnected locally compact group, $\phi:G\to G$ a continuous endomorphism and $\mathcal B\subseteq \mathcal B(G)$ a local base at $e_G$. Then $h_{top}(\phi)=\sup\{H_{top}(\phi,U):U\in\mathcal B\}$.
\end{claim}

Applying this claim on topological entropy, as well as Theorem \ref{TP} and Theorem \ref{lf}, in the following proposition we give a sufficient condition to have equality in \eqref{geq}. 

\begin{proposition}\label{nubb}
Let $G$ be a totally disconnected locally compact group and $\phi:G\to G$ a topological automorphism. If $\nub(\phi)=\{e_G\}$ then $\log s(\phi)=h_{top}(\phi)$.
\end{proposition}
\begin{proof}
Suppose that $\nub(\phi)=\{e_G\}$, then for every $U\in\mathcal B(G)$ tidy for $\phi$ we have $\log s(\phi)=\log[\phi(U_+):U_+]$ by Theorem \ref{TP} and so $H_{top}(\phi,U)=\log s(\phi)$ by Theorem \ref{lf}. 
We are assuming that $\nub(\phi)=\{e_G\}$, so the tidy subgroups form a local base at $e_G$ by Corollary \ref{tidylocalbase}. Hence Claim \ref{basesuff} permits to conclude that $\log s(\phi)=h_{top}(\phi)$.
\end{proof}

In particular, Proposition \ref{nubb} says that, if the inequality $\log s(\phi)\leq h_{top}(\phi)$ in \eqref{geq} is strict, then $\nub(\phi)\neq\{e_G\}$. As already mentioned, this is the case of topological automorphisms $\phi$ of totally disconnected compact groups with positive topological entropy; indeed, $\log s(\phi)=0$, while $h_{top}(\phi)>0$. So we have the following consequence of Proposition \ref{nubb} on topological entropy.

\begin{corollary}
Let $G$ be a totally disconnected compact group and $\phi:G\to G$ a topological automorphism. If $h_{top}(\phi)>0$ then $\nub(\phi)\neq\{e_G\}$.
\end{corollary}

Another consequence of Proposition \ref{nubb} on topological entropy concerns its values. Indeed, the scale function assumes only finite values as noted above, while the topological entropy can be infinite, being defined as a supremum. We see now that when the nub is trivial, the topological entropy has only finite values.

\begin{corollary}
Let $G$ be a totally disconnected locally compact group and $\phi:G\to G$ a topological automorphism. If $\nub(\phi)=\{e_G\}$, then $h_{top}(\phi)$ is finite. 

Moreover, $h_{top}(\phi)=H_{top}(\phi,U)=\log[\phi(U_+):U_+]$ for every $U\in\mathcal B(G)$ tidy for $\phi$.
\end{corollary}
\begin{proof}
By Proposition \ref{nubb} we have $h_{top}(\phi)=\log s(\phi)$. Then apply Theorems \ref{TP} and \ref{lf}.
\end{proof}

We give now two examples of non-compact totally disconnected locally compact groups $G$ and topological automorphisms $\phi:G\to G$ for which $\nub(\phi)=\{e_G\}$ and so $\log s(\phi)=h_{top}(\phi)$.

\begin{example}\label{esugpadica}
\begin{itemize}
\item[(a)] For any integer $n>0$ and every topological automorphism $\phi\colon \Q_p^n\to\Q_p^n$ the equality $\log s(\phi)=h_{top}(\phi)$
holds true. Indeed, we know that $\nub(\phi)$ is trivial by Example \ref{June13}(f), so we can conclude using Proposition \ref{nubb}.
\item[(b)] Let $p$ be a prime, $G = \Z(p)^\Z$ and $\sigma: G \to G$ the left Bernoulli shift. Modify the usual compact product topology of $G$ taking $U = \Z(p)^{\N}$ to be an open subgroup (equipped with its compact product topology) of $G$ in this new topology. With respect to Example \ref{beta}, the value of the topological entropy remains $h_{top}(\sigma)=\log p$.

Since $\nub(\sigma)$ is trivial, Proposition \ref{nubb} applies to give $\log s(\sigma)=h_{top}(\sigma) = \log p >0$ in this case (compare with the particular case of coincidence of $\log s(\phi)$ and $h_{top}(\phi)$ considered in Corollary \ref{s=1} below).
\end{itemize}
\end{example}

We know that $s(\phi) = 1$ if and only if there exists a $\phi$-invariant $U \in \mathcal B(G)$ by Lemma \ref{invariantsubgroup}(d). Moreover, Corollary \ref{tidylocalbase}(b) implies that if $s(\phi)=1$ and $\nub(\phi) = \{e_G\}$ then $G$ has a local base at $e_G$ consisting of $\phi$-invariant compact open subgroups. We see in the next corollary that this condition is equivalent to $h_{top}(\phi)=0$.

\begin{corollary}\label{s=1}
Let $G$ be a totally disconnected locally compact group and $\phi:G\to G$ a topological automorphism. Then the following conditions are equivalent: 
\begin{itemize}
\item[(a)] $h_{top}(\phi)=0$; 
\item[(b)] $G$ has a local base at $e_G$ formed by $\phi$-invariant $U \in \mathcal B(G)$; 
\item[(c)] $G$ has a local base at $e_G$ formed by subgroups tidy for $\phi$ and $s(\phi)=1$; 
\item[(d)] $\nub(\phi)=\{e_G\}$ and $s(\phi)=1$.
\end{itemize}
\end{corollary}
\begin{proof}
(a)$\Rightarrow$(b) By the definition of topological entropy $h_{top}(\phi)=0$ implies $H_{top}(\phi,U)=0$ for every $U\in\mathcal B(G)$. The condition $H_{top}(\phi,U)=0$ implies that there exists an integer $n\geq0$ such that $U_{-n}=U_-$; this follows from \cite[Lemma 3.1]{DG-limitfree} in the compact case (see \cite{DG-htop} for the general case). Then $U_-\in\mathcal B(G)$, $U_-\subseteq U$ and it is $\phi$-invariant. This shows that $G$ has a local base at $e_G$ formed by $\phi$-invariant $U \in \mathcal B(G)$.

Now (b)$\Rightarrow$(c)$\Rightarrow$(d) are obvious, and (d)$\Rightarrow$(a) follows from Proposition \ref{nubb}. 
\end{proof}

The hypothesis $s(\phi)=1$ of Corollary \ref{s=1}(c,d) is satisfied in obvious way when $G$ is compact. In contrast to Example \ref{beta}, now Example \ref{example} furnishes a totally disconnected locally compact group $G$ that is not compact, and a topological automorphism $\phi:G\to G$ such that $s(\phi)=1$ and $h_{top}(\phi)>0$. By Proposition \ref{nubb}, this yields that $\nub(\phi)$ is necessarily a non-trivial subgroup of $G$, and in this case $\nub(\phi)$ is also proper (compare with Examples \ref{beta} and \ref{esugpadica}(b)).

\begin{example}\label{example}
Let $p$ be a prime and $G=\Z(p^\infty)^{\Z}$. Imposing that $U=\Z(p)^\Z$ is open in $G$ (equipped with its compact product topology), then $G$ is given a totally disconnected locally compact (non-compact) topology. Consider $\sigma:G\to G$ the left Bernoulli shift;
clearly $\sigma(U)=U$, and then:
\begin{itemize}
\item[(a)] $\nub(\sigma)=U$;
\item[(b)] $s(\sigma)=1$;
\item[(c)] $H_{top}(\sigma,U)=0$;
\item[(d)] $H_{top}(\sigma,V)=\log p$, where $V=\Z(p)^{-\N_+}\oplus\{0\}\oplus\Z(p)^{\N_+}$; in fact $[\sigma(V_+):V_+]=p$ and apply Theorem \ref{lf}.
Note that $V_+=\Z(p)^{\N_+}$ and $V_-=\Z(p)^{-\N_+}$, therefore $V$ is tidy above for $\sigma$. On the other hand, $V_{++}=\Z(p)^{(-\N)}\oplus\Z(p)^{\N_+}$, which is dense in $U$ and so it is not closed; in other words $V$ is not tidy below for $\sigma$. 
\item[(e)] $h_{top}(\sigma)=\log p$, since $\{V_n:n\in\Z\}$ is a local base at $e_G$, $H_{top}(\sigma,V_n)=\log p$ as in item (d). Then apply Claim \ref{basesuff}.
\end{itemize}
\end{example}

Since the sufficiency was already proved in Proposition \ref{nubb}, it is enough to prove the necessity in order to complete the proof of Theorem \ref{nubbanale}.

\smallskip 

\begin{proof}[\bf First proof of Theorem \ref{nubbanale}]
Suppose that $\nub(\phi)\neq\{e_G\}$; then there exists an element $e_G\neq g\in \nub(\phi)$. The family
$$\mathcal{B}_g=\{U\in\mathcal{B}(G): U \text{ is tidy above for }\phi,\, g\notin U \}$$
is a local base at $e_G$, so by Claim \ref{basesuff}
$$h_{top}(\phi)=\sup\{H_{top}(\phi,U):U\in\mathcal{B}_g\}.$$
As every subgroup tidy for $\phi$ contains $g$ by the choice of $g$, no subgroup $U$ in $\mathcal{B}_g$ is tidy for $\phi$.
By Theorem \ref{TP}, no $U\in\mathcal B_g$ is minimizing for $\phi$ and so, in view of Theorem \ref{lf} and Corollary \ref{s=min}, 
$$\log s(\phi)<\log s(\phi,U)=\log[\phi(U_+):U_+]=H_{top}(\phi,U).$$ 
Therefore, $\log s(\phi)<h_{top}(\phi)$.  
\end{proof}

\smallskip

\begin{proof}[\bf Second proof of Theorem \ref{nubbanale}]
Suppose that $\log s(\phi)=h_{top}(\phi)$. Let $U\in \mathcal B(G)$ be tidy above for $\phi$.
Then $$s(\phi,U)=[\phi(U):U\cap \phi(U)]=[\phi(U_+):U_+]$$ by Corollary \ref{s=min}. By Theorem \ref{lf} and by our assumption we have 
$$\log s(\phi,U)=\log[\phi(U_+):U_+]=H_{top}(\phi,U)\leq h_{top}(\phi)=\log s(\phi).$$ Therefore $s(\phi)=s(\phi,U)$ and so $U$ is minimizing for $\phi$, that is $U$ is tidy for $\phi$ by Theorem \ref{TP}. We have shown that every $U\in\mathcal B(G)$ that is tidy above for $\phi$ is also tidy for $\phi$. Then $\mathcal M(G,\phi)$ is a local base at $e_G$ by Corollary \ref{esiste n}, hence $\nub(\phi)=\{e_G\}$ by Corollary \ref{tidylocalbase}.
\end{proof}

\section{Basic ``entropic'' properties of the scale function}\label{ep}

In this section we give properties of the scale function similar to the basic properties satisfied by the topological entropy; so we start reminding the latter ones in the following result.

\begin{fact}\label{properties}
Let $G$ be a totally disconnected locally compact group and $\phi:G\to G$ a topological automorphism.
\begin{itemize}
\item[(a)]\emph{[Logarithmic Law]} For every integer $k\geq0$ we have $h_{top}(\phi^k) = k \cdot h_{top}(\phi)$.
\item[(b)]\emph{[Invariance under conjugation]} If $H$ is another totally disconnected locally compact group and $\xi:G\to H$ is a topological isomorphism, then $h_{top}(\phi) = h_{top}(\xi\phi\xi^{-1})$.
\item[(c)]\emph{[Monotonicity]} If $H$ is a $\phi$-stable closed subgroup of $G$, then $h_{top}(\phi)\geq h_{top}(\phi\restriction_H)$; if $H$ is normal and $\overline\phi:G/H\to G/H$ is the topological automorphism induced by $\phi$, then $h_{top}(\phi)\geq h_{top}(\overline{\phi})$.
\item[(d)]\emph{[Weak Addition Theorem]} If $G=G_1\times G_2$ and $\phi_i:G_i\to G_i$ is a topological automorphism for $i=1,2$, then $h_{top}(\phi_1\times\phi_2)=h_{top}(\phi_1)+h_{top}(\phi_2)$.
\item[(e)]\emph{[Continuity]} If $G$ is an inverse limit $G=\varprojlim G/N_i$ with $N_i$ a $\phi$-stable closed normal subgroup, then $h_{top}(\phi)=\sup_{i\in I}h_{top}(\overline\phi_i)$, where $\overline\phi_i:G/N_i\to G/N_i$ is the topological automorphism induced by $\phi$.
\end{itemize}
\end{fact}

The Logarithmic Law for the scale function is already known:

\begin{fact}[Logarithmic Law]\emph{\cite[Corollary 3]{Willis}}\label{loglaw}
Let $G$ be a totally disconnected locally compact group, $\phi:G\to G$ a topological automorphism and $n\geq 0$ an integer. Then $s(\phi^n)=s(\phi)^n$.
\end{fact}

Invariance under conjugation is clear also for the scale function:

\begin{lemma}[Invariance under conjugation]\label{invariance}
Let $G$ be a totally disconnected locally compact group and $\phi:G\to G$ a topological automorphism. Let $H$ be another totally disconnected locally compact group and $\xi:G\to H$ a topological isomorphism. Then
\begin{itemize}
\item[(a)] $U\in\mathcal B(G)$ if and only if $\xi(U)\in\mathcal B(H)$, therefore $\mathcal{B}(H)=\{\xi(U): U \in \mathcal{B}(G)\}$.
\item[(b)] If $U\in\mathcal B(G)$, then $U\in\mathcal M(G,\phi)$ if and only if $\xi(U)\in\mathcal M(H,\xi\phi\xi^{-1})$, and in particular, 
$$\mathcal M(H,\xi\phi\xi^{-1})=\{\xi(U):U\in\mathcal M(G,\phi)\};$$
\item[(c)] $s(\phi)=s(\xi\phi\xi^{-1})$.
\end{itemize}
\end{lemma}
\begin{proof}
(a) is clear, (b) follows from the fact that $s_G(\phi,U)=s_H\bigl(\xi\phi\xi^{-1},\xi(U)\bigr)$ for every $U\in\mathcal{B}(G)$, and (c) follows from (b).
\end{proof}

Consider the case $H=G$ in the above lemma. We see in Example \ref{minex} that, while $\mathcal M(G,\xi\phi\xi^{-1})=\{\xi(U):U\in\mathcal M(G,\phi)\}$ and also $s(\phi)=s(\xi\phi\xi^{-1})$ by Lemma \ref{invariance}(c), it may occur the case that $\mathcal M(G,\phi)$ do not coincide with $\mathcal M(G,\xi\phi\xi^{-1})$.
This stresses the fact that the correspondence between minimizing subgroups for $\phi$ and minimizing subgroups for $\xi\phi\xi^{-1}$ is given by $U\mapsto \xi(U)$.

\begin{example}\label{minex}
Let $\phi\colon \Q_p^2\to\Q_p^2$ the topological automorphism defined by the matrix $\begin{pmatrix}0&p\\p^{-1}&0\end{pmatrix}$. 
Then $\phi^2=id$ and so $s(\phi)=1$ by Fact \ref{loglaw}.
Nevertheless,
$$s(\phi,\J_p^2)=\bigl[(p\J_p)\times(p^{-1}\J_p):(p\J_p)\times\J_p\bigr]=p$$
and hence $\J_p^2$ is not a minimizing subgroup for $\phi$, although it is a minimizing subgroup for the canonical Jordan form of 
$\phi$. Indeed, let $\xi\colon \Q_p^2\to\Q_p^2$ be the topological automorphism defined by $\begin{pmatrix}p&-p\\1&1\end{pmatrix}$ and $\psi\colon \Q_p^2\to\Q_p^2$ the topological automorphism defined by $\begin{pmatrix}1&0\\0&-1\end{pmatrix}$, then $\phi=\xi\psi\xi^{-1}$, i.e., $\psi$ is the canonical Jordan form of $\phi$.
It is obvious that $s(\psi,\J_p^2)=1$.
\end{example}

Monotonicity was proved in \cite{Willis2}, indeed the following more precise relation was given there.

\begin{fact}[Monotonicity]\emph{\cite[Proposition 4.7]{Willis2}}\label{closedsubgroups}
Let $G$ be a totally disconnected locally compact group, $\phi:G\to G$ a topological automorphism and $H$ a $\phi$-stable closed subgroup of $G$. Then
\begin{itemize}
\item[(a)] $s(\phi)\geq s(\phi\restriction_H)$.
\end{itemize}
If $H$ is also normal and $\overline\phi:G/H\to G/H$ is the topological automorphism induced by $\phi$, then
\begin{itemize}
\item[(b)] $s(\phi\restriction_H)\cdot s(\overline\phi)$ divides $s(\phi)$.
\end{itemize}
\end{fact}

\begin{remark}
\begin{itemize}
\item[(a)] We call the property in item (d) of Fact \ref{properties} Weak Addition Theorem. Indeed, the stronger so-called Addition Theorem holds for the topological entropy in the compact case (see \cite{B,St,Y}). More precisely, by Addition Theorem we mean the following property, imposed on all continuous endomorphisms $\phi:G\to G$ of compact groups $G$: 
if $H$ is a closed $\phi$-invariant normal subgroup of $G$, then 
$$
h_{top}(\phi)=h_{top}(\phi\restriction_H)+h_{top}(\overline\phi),
$$
where $\overline\phi:G/H\to G/H$ is the continuous endomorphism induced by $\phi$.
\item[(b)] It is not known whether the Addition Theorem for the topological entropy holds also in the general case of locally compact groups, even under the hypotheses that $G$ is totally disconnected (and abelian) and that $\phi:G\to G$ is a topological automorphism.
\item[(c)] The counterpart of the Addition Theorem for the scale function does not hold true in general, since \cite[Example 6.4]{Willis2} shows that the inequality $s(\phi)\geq s(\phi\restriction_H)\cdot s(\overline\phi)$ in Fact \ref{closedsubgroups}(b) can be strict.

On the other hand, we see in Theorem \ref{wATscala} below that the Weak Addition Theorem holds also for the scale function.
\end{itemize}
\end{remark}

Note that we call this kind of properties Addition Theorem also for the scale function, even if they have a multiplicative form in this case; just take the logarithm to have the additive form.

\begin{theorem}[Weak Addition Theorem]\label{wATscala}
Let $G$, $H$ be totally disconnected locally compact groups and $\phi:G\to G$, $\psi:H\to H$ topological automorphisms. Then $s(\phi\times\psi)=s(\phi)\cdot s(\psi)$.
\end{theorem}
\begin{proof}
Let $V\in \mathcal B(G)$ be tidy for $\phi$ and $W\in \mathcal B(H)$ be tidy for $\psi$. For the compact and open subgroup $V\times W\subseteq G\times H$, we have that
\begin{align*}
(V\times W)_+&=\bigcap_{k\geq 0}(\phi\times\psi)^k(V\times W)=\bigcap_{k\geq 0}\bigl(\phi^k(V)\times 
\psi^k(W)\bigr)\\&=\Bigl(\bigcap_{k\geq 0}\phi^k(V)\Bigr)\times \Bigl(\bigcap_{k\geq 0}\psi^k(W)\Bigr)=V_+\times W_+,
\end{align*}
and in the same way one can prove that $(V\times W)_-=V_-\times W_-$. Since 
$$
(V\times W)_+(V\times W)_-=(V_+V_-)\times (W_+W_-)=V\times W,
$$
we have that $V\times W$ is tidy above for $\phi\times \psi$. The subgroup $V\times W$ is also tidy below for $\phi\times \psi$ because
\begin{align*}
(V\times W)_{++}&=\bigcup_{k\geq 0}(\phi\times\psi)^k(V\times W)_+
=\bigcup_{k\geq 0}\Bigl(\phi^k(V_+)\times \psi^k(W_+)\Bigr)\\
&\stackrel{*}{=}\Bigl(\bigcup_{k\geq 0}\phi^k(V_+)\Bigr)\times \Bigl(\bigcup_{k\geq 0}\psi^k(W_+)\Bigr)=
V_{++}\times W_{++}
\end{align*}
is a closed subgroup of $G\times H$. The equality ($*$) holds because the families $\{\phi^k(V_+)\}_{k \geq 0}$ and $\{\psi^k(W_+)\}_{k\geq 0}$ are increasing families of subgroups of $G$ and $H$ respectively.

Therefore $V\times W$ is tidy for $\phi\times \psi$ and
$$
s(\phi\times \psi)=\bigl[(\phi\times \psi)(V\times W):\bigl((V\times W)\cap (\phi\times \psi)(V\times W)\bigr)\bigr].
$$
This index is equal to
$$
\bigl[\phi (V):V\cap \phi (V)\bigr]\cdot\bigl[\psi (W):W\cap \psi (W)\bigr]=s(\phi)\cdot s(\psi)
$$
and hence $s(\phi\times \psi)=s(\phi)\cdot s(\psi).$
\end{proof}

We conclude this section discussing the continuity of the scale function with respect to direct and inverse limits. The next proposition should be compared with \cite[Proposition 5.3]{Willis2}, where the scale function is considered for inner automorphisms on an increasing sequence of closed subgroups.

\begin{proposition}[Continuity for direct limits]\label{directlimits}
Let $G$ be a totally disconnected locally compact group and $\phi\colon G\to G$ a topological automorphism. If $G\cong \varinjlim_{i\in I} H_i$, where $\{H_i\}_{i\in I}$ is a directed system of $\phi$-stable open subgroups of $G$, then there exists $j\in I$ such that
$$s(\phi)=s(\phi\restriction_{H_j})=\max_{i\in I}s(\phi\restriction_{H_i}).$$
\end{proposition}
\begin{proof} 
By Lemma \ref{closedsubgroups}(a) the inequalities 
\begin{equation}\label{inequality1}
s(\phi)\geq s(\phi\restriction_{H_i})
\end{equation}
hold true for every $i\in I$. Then $s(\phi)\geq\max_{i\in I}s(\phi\restriction_{H_i})$.

Let $U\in\mathcal M(G,\phi)$. Then $\{H_i\cap U\}_{i\in I}$ is an open covering of $U$ and so it admits a finite open subcover, because $U$  is compact. This means that there exists a finite set $F\subseteq I$ such that $U\subseteq \bigcup_{i\in F}H_i.$
Moreover, there exists an index $j\in I$ such that $H_i\subseteq H_j$ for every $i\in F$ and so $U\subseteq H_j$.
In particular, $U\in \mathcal B(H_j)$. This implies that $U$ is tidy above also for $\phi\restriction_{H_j}$. Indeed, both automorphisms 
$\phi$ and $\phi\restriction_{H_j}$ share the same subgroups $U_+$ and $U_-$. Moreover,  $U$ is tidy below for $\phi\restriction_{H_j}$, because 
$$
U_{\phi\restriction_{H_j}++}=U_{\phi++}=U_{\phi++}\cap H_j
$$
is a closed subgroup of $H_j$. This means that
$$
s(\phi)=s(\phi,U)=s(\phi\restriction_{H_j},U)= s(\phi\restriction_{H_j}).
$$
Hence, in view of \eqref{inequality1},
$$s(\phi)=s(\phi\restriction_{H_j})=\max_{i\in I}s(\phi\restriction_{H_i}),$$
and this concludes the proof.
\end{proof}

\begin{remark}   A counterpart of Proposition \ref{directlimits} regarding continuity for inverse limits holds true:  if $G$ is a totally disconnected locally compact group, $\phi:G\to G$ is a topological automorphism and $G=\varprojlim G/N_i$ is an inverse limit where $\{N_i\}_{i\in I}$ is an inverse system of $\phi$-stable closed normal subgroups of $G$, then $s(\phi)=\max_{i\in I} s(\overline \phi_i)$, where $\overline \phi_i:G/N_i\to G/N_i$ is the topological automorphism induced by $\phi$ for every $i\in I$. A proof can be found in \cite{BDGT}, while the case of inner automorphisms is proved in \cite[Proposition 5.4]{Willis2}.  
\end{remark}

Applying Proposition  \ref{directlimits} one obtains the following corollary still concerning continuity of the scale function with respect to direct limits; the condition on the stable subgroups is relaxed from open to closed, while the set of indices is now supposed to be countable.

\begin{corollary}
Let $G$ be a totally disconnected locally compact group and $\phi\colon G\to G$ a topological automorphism. If $G\cong \varinjlim_{n\geq 0} H_n$, where $\{H_n\}_{n\geq 0}$ is a directed system of $\phi$-stable closed subgroups of $G$, then exists an integer $n\geq0$ such that
$$s(\phi)=s(\phi\restriction_{H_n})=\max_{n\in \N}s(\phi\restriction_{H_n}).$$
\end{corollary}

\begin{proof}  Apply the Baire Category Theorem to $G=\bigcup_{n\in\N}H_n$ to conclude that for some $m\geq0$ the subgroup $H_m$ has non-empty  interior. Therefore, $H_n$ is open for all $n \geq m$.  Now apply Proposition \ref{directlimits} to the family of these open subgroups.
\end{proof}

\section{Scalar $p$-adic Yuzvinski Formula and Bridge Theorem}\label{yf-bt}

In the first part of this section, and more precisely in Theorem \ref{formulaQpN}, we compute directly the value of the scale function of any topological automorphism $\phi:\Q_p^n\to \Q_p^n$, where $n>0$; note that such an automorphism $\phi$ is a $\Q_p$-linear transformation and so it is given by an $n\times n$ matrix with coefficients in $\Q_p$.

\medskip
We start recalling some useful information about the $p$-adic numbers, giving reference to \cite{G} for more details. Let $\lvert -\rvert_p$ be the $p$-adic norm over $\Q_p$, that is, for $\xi\in\Q_p$
$$\lvert\xi\rvert_p=\begin{cases}
0&\text{if }\xi=0,\\
p^r&\text{if }\xi=p^{-r}\Bigl(\sum_{i=0}^{\infty}a_ip^i\Bigr)\text{ with }a_0\in\{1,2, \ldots, p-1\}.
\end{cases}$$
Note that $\J_p=\{\xi\in\Q_p: \lvert\xi\rvert_p\leq 1\}$ is a local PID with maximal ideal $\{\xi\in\Q_p: \lvert\xi\rvert_p< 1\}$.

\smallskip
If $K$ is a finite extension of $\Q_p$ of degree $d=[K:\Q_p]$, then the $p$-adic norm $\lvert-\rvert_p$ over $\Q_p$ can be extended 
to a norm over $K$, and this extension is unique. We indicate this extended norm with $\lvert-\rvert_p$ and we call it the $p$-adic 
norm over $K$. Let $$\mathcal{O}=\{\xi\in K: \lvert\xi\rvert_p\leq 1\},$$ then $\mathcal{O}$ is a local PID with maximal ideal
$$
\mathfrak{m}=\{\xi\in K: \lvert\xi\rvert_p< 1\}.
$$
Consider now a generator $\pi'$ for $\mathfrak{m}$, then there exists an integer $e>0$ such that $p=u\pi'^e$, where $u$ is a unit in $\mathcal{O}$.
 Denote by $\pi$ a generator of $\mathfrak{m}$ such that $p=\pi^e$. This number $e$ is independent  of the choice of the generator of $\mathfrak{m}$, divides the degree $d$ of the extension and it is called the \emph{ramification index} of $K$ over $\Q_p$. One can prove that the residual field $\mathcal{O}/\mathfrak{m}$ has cardinality
\begin{equation}\label{residualfield}
\bigl\lvert\mathcal{O}/\mathfrak{m}\bigr\rvert=[\mathcal{O}:\mathfrak{m}]=p^f,
\end{equation}
where we let $f=d/e$.

\medskip
The following example is the basic case necessary to obtain in Theorem \ref{formulaQpN} an explicit formula for the scale function of topological automorphisms of $\Q_p^n$.

\begin{example}\label{exJordan}
Let $K$ be a finite extension of $\Q_p$ of degree $d$, let $\lambda\in K$ and consider the topological automorphism $\phi\colon K^n
\to K^n$ defined by the Jordan block
$$
J=\begin{pmatrix}
\lambda&1&0&\dots&0\\0&\lambda&1&\dots&0\\\vdots&\ddots&\ddots&\ddots&\vdots\\0&\dots&0&\lambda&1\\0&\dots&0&0&\lambda
\end{pmatrix}
\in GL_n(K).
$$
We see that $\mathcal{O}^n$ is minimizing for $\phi$ and that
\begin{equation}\label{detJ}
s_{K^n}(\phi)=\max\{1,\lvert\lambda\rvert_p^{nd}\}=\max\{1,\lvert\det J\rvert_p^d\}.
\end{equation}
Note that the nature of the automorphism $\phi$ is completely determined by $\lambda$. In fact if $\lambda\in\mathcal{O}$ then $\mathcal{O}^n$ is $\phi$-invariant and $s(\phi)=s(\phi,\mathcal{O}^n)=1$, otherwise it is inversely $\phi$-invariant. In both cases the subgroup is minimizing for $\phi$ by Lemma \ref{invariantsubgroup}(a,c).
 
Suppose that $\lambda\notin\mathcal{O}$, this means that
\begin{equation*}
s(\phi)=s(\phi,\mathcal{O}^n)=\bigl[\phi(\mathcal{O}^n):\mathcal{O}^n\cap\phi(\mathcal{O}^n)\bigr]=\bigl[(\lambda\mathcal{O})^n:\mathcal{O}^n
\bigr]=\bigl[\lambda\mathcal{O}:\mathcal{O}\bigr]^n.
\end{equation*}
Let $e$ be the ramification index of the extension $K$ over $\Q_p$, $\pi$ a generator for $\mathfrak{m}$ such that $p=\pi^e$ and $f=d/e$. Then $\lambda=\pi^{-l}\xi$, where $\xi\in \mathcal{O}\setminus \mathfrak{m}$, $l>0$ and
\begin{equation*}
\lvert\lambda\rvert_p=\lvert\pi^{-l}\rvert_p=\lvert p^{-l/e}\rvert_p=p^{l/e}.
\end{equation*}
This yields
$$
s(\phi)=\bigl[\lambda\mathcal{O}:\mathcal{O}\bigr]^n=\bigl[\pi^{-l}\mathcal{O}:\mathcal{O}\bigr]^n=[\pi^{-1}\mathcal{O}:\mathcal{O}
]^{ln}.
$$
By \eqref{residualfield}, $[\pi^{-1}\mathcal{O}:\mathcal{O}]=p^f$ and so
$$
s(\phi)=(p^f)^{ln}=(p^{l/e})^{efn}=\lvert\lambda\rvert_p^{dn}.
$$
This means exactly that
$$
s_{K^n}(\phi)=\max\{1,\lvert \lambda\rvert_p^{dn}\}.
$$
In conclusion note that $\lvert \lambda\rvert_p^n=\lvert\det J\rvert_p$, so this proves also the other equality in \eqref{detJ}. 
\end{example}

By Example \ref{esugpadica}(a) we have that the equality $\log s(\phi)=h_{top}(\phi)$ holds true for all topological automorphism $\phi\colon \Q_p^n\to\Q_p^n$, so one could apply the $p$-adic Yuzvinski Formula for the topological entropy proved in \cite{LW} to obtain \eqref{Yuz}. Nevertheless, we give another direct proof of this formula for sake of completeness, but also because the computation of the scale function is simpler than that of the topological entropy; indeed, for the scale function it suffices to take into account only one compact open subgroup which is minimizing (i.e., tidy) for $\phi$, without any recourse to the Haar measure.

\begin{theorem}\label{formulaQpN}
Let $\phi\colon \Q_p^n\to\Q_p^n$ be a topological automorphism, for an integer $n>0$. Then
\begin{equation}\label{Yuz}
s_{\Q_p^n}(\phi)=\prod_{\lvert\lambda_i\rvert_p>1}\lvert\lambda_i\rvert_p,
\end{equation}
where $\{\lambda_1,\ldots,\lambda_n\}$ is the family of all eigenvalues of $\phi$ contained in a finite extension $K$ of $\Q_p$.
\end{theorem}
\begin{proof}

Assume without loss of generality that $K=\Q_p[\lambda_1,\dots,\lambda_n]$, that is, $K$ is the splitting field of the minimal 
polynomial of $\phi$ over $\Q_p$, and let $d=[K:\Q_p]$.

Let $\phi^K=\phi\otimes_{\Q_p} id_K\colon K^n\to K^n$, where $\otimes_{\Q_p}$ is the tensor product over $\Q_p$. The automorphisms 
$\phi$ and $\phi^K$ are represented by the same matrix respectively over $\Q_p$ and $K$, hence they have the same eigenvalues.

Let $\mathcal{A}$ be a base of $K$ over $\Q_p$, then every $\xi\in K$ has coordinates $[\xi]_{\mathcal{A}}=(\xi_{(1)},\dots,
\xi_{(d)})$. Moreover, $K^n\cong \Q_p^{dn}$ and this isomorphism $\alpha\colon K^n\to \Q_p^{dn}$ is given by
$$
\alpha(\xi_1,\dots,\xi_n)=\Bigl([\xi_1]_{\mathcal{A}},\dots,[\xi_n]_{\mathcal{A}}\Bigr).
$$
Let
$$
\Phi=\underbrace{\phi\times\dots\times\phi}_{d\text{ times}}\colon \Q_p^{dn}\to \Q_p^{dn};
$$
then $\phi^K=\alpha^{-1}\Phi\alpha$. Lemma \ref{invariance}(c) and Fact \ref{loglaw} yield
\begin{equation}
s_{\Q_p^n}(\phi)^d=s_{\Q_p^{dn}}(\Phi)=s_{K^n}(\phi^K).\label{eq1}
\end{equation}
In $K^n$ there exists a base with respect to which the automorphism $\phi^K$ is in the canonical Jordan form, because $\phi^K$
splits over $K$. So, by Lemma \ref{invariance}(c), we can suppose that $\phi^K$ itself is represented by a matrix in the canonical 
Jordan form. Let $J_1,\dots, J_r$ be the Jordan blocks of this matrix, where each $J_l\in GL_{n_l}(K)$ is associated to the eigenvalue 
$\xi_l\in \{\lambda_1,\dots,\lambda_n\}$ of $\phi^K$.

Define $\phi_l\colon K^{n_l}\to K^{n_l}$ as the topological  automorphism associated to $J_l$, for every $l=1,\dots, r$.  Then by Example \ref{exJordan}
\begin{equation}
s_{K^{n_l}}(\phi_l)=\max\{1,\lvert\xi_l\rvert_p^{dn_l}\}=\max\{1,\lvert\det J_l\rvert_p^d\}.\label{eq2}
\end{equation}
Since $\phi^K = \phi_1\times \ldots \times \phi_r$, Theorem \ref{wATscala} entails
\begin{equation*}
s_{K^n}(\phi^K)=s_{K^{n_1}}(\phi_1)\cdot\ldots\cdot s_{K^{n_r}}(\phi_r).
\end{equation*}
Hence, from \eqref{eq1} and \eqref{eq2}, we obtain 
$$s_{\Q_p^n}(\phi)^d=s_{K^n}(\phi^K)=\prod_{\substack{l=1\\\lvert\xi_l\rvert_p>1}}^r\lvert\xi_l\rvert_p^{dn_l}.$$
This means that
$$s_{\Q_p^n}(\phi)=\prod_{\substack{l=1\\\lvert\xi_l\rvert_p>1}}^r\lvert\xi_l\rvert_p^{n_l}=\prod_{\substack{i=1\\\lvert\lambda_i
\rvert_p>1}}^n\lvert\lambda_i\rvert_p=\prod_{\lvert\lambda_i\rvert_p>1}\lvert\lambda_i\rvert_p,$$
and this concludes the proof.
\end{proof}

\medskip
In the second part of this section we provide a so-called Bridge Theorem for the scale function. To this end we first recall some properties of the Pontryagin duality.

\smallskip
We say that a locally compact group $G$ is \emph{compactly covered}, if every element of $G$ is contained in some compact subgroup of $G$.  The following folklore fact can be easily deduced from the standard properties of locally compact abelian groups. 

\begin{fact}\label{factCC} For a locally compact abelian group  $G$, the following conditions are equivalent:
\begin{itemize}
  \item[(a)] $G$ is compactly covered; 
  \item[(b)] $G$ contains no copies of the discrete group $\Z$; 
  \item[(c)] there exist no continuous surjective homomorphisms $\widehat{G}\to \T$; 
  \item[(d)] $\widehat{G}$ is totally disconnected. 
\end{itemize}
\end{fact}

Our interest in Fact \ref{factCC} stems from the necessity to describe the class of totally disconnected locally compact abelian groups $G$, such that their dual group $\widehat{G}$ is totally disconnected as well. According to the above fact, these are precisely the compactly covered totally disconnected locally compact abelian groups. Therefore, for such a group $G$ one can define the scale function on the dual group $\widehat{G}$, and we are interested in the relationship between $s_G(\phi)$ and $s_{\widehat{G}}(\widehat{\phi})$, where $\phi\colon G\to G$ is a topological automorphism of $G$.

\medskip
In the next fact we collect several known properties that apply in the proof of Theorem \ref{BTS}.

\begin{fact}\label{fattiduale}
Let $G$ be a locally compact abelian group, $\phi\colon G\to G$ a topological automorphism and $U$ a compact subgroup of $G$. Then:
\begin{itemize}
  \item[(a)] $(U^\bot)^\bot=U$;
  \item[(b)] $U^\bot\cong \widehat{G/U}$ and $\widehat{U}\cong\widehat{G}/U^\bot$;
  \item[(c)] $\bigl(U+ \phi(U)\bigr)^\bot=U^\bot \cap\widehat{\phi}^{-1}(U^\bot)$;
  \item[(d)] if $V$ is another compact subgroup of $G$ and $U\subseteq V$, then $\widehat{V/U}\cong U^\bot/V^\bot$.
\end{itemize}
\end{fact}

We are now in the conditions to prove the next Bridge Theorem, which asserts in particular that $s_{\widehat{G}}(\widehat{\phi})$ is equal to $s_G(\phi)$; note that according to Fact \ref{factCC}, the group $\widehat{G}$ is totally disconnected, so one can define the scale function on $\widehat{G}$. 

\begin{theorem}\label{BTS}
Let $G$ be a compactly covered totally disconnected locally compact abelian group and $\phi \colon G \to G$ a topological automorphism.
Then:
\begin{itemize}
  \item[(a)] $U\in\mathcal B(G)$ if and only if $U^\perp\in\mathcal B(\widehat G)$;
  \item[(b)] $s(\phi,U)=s(\phi,U^\perp)$;
  \item[(c)] $U\in \mathcal M(G,\phi)$ if and only if $U^\perp\in \mathcal M(\widehat G,\widehat\phi)$;
  \item[(d)] $s_{\widehat{G}}(\widehat{\phi}) = s_G(\phi)$.
\end{itemize}
\end{theorem}

\begin{proof} 
(a) Assume that $U\in\mathcal B(G)$. Then $G/U$ is discrete because $U$ is open in $G$, and so $\widehat{G/U}$ is a compact group. Since $U^\bot\cong \widehat{G/U}$ by Fact \ref{fattiduale}, so $U^\bot$ is a compact subgroup of $\widehat{G}$. Moreover, $U$ is compact in $G$ and so, since $\widehat G/U^\perp\cong \widehat U$ by Fact \ref{fattiduale}, $\widehat G/U^\perp$ is discrete; therefore, $U^\bot$ is open in $\widehat{G}$. Hence, we have proved that $U\in\mathcal B(G)$ implies $U^\perp\in\mathcal B(\widehat G)$. 

To verify the converse implication it suffices to note that by Pontryagin duality $G$ is canonically isomorphic to $\widehat{\widehat{G}}$, so that we can identify $G$ and $\widehat{\widehat{G}}$; moreover, $(U^\perp)^\perp=U$ by Fact \ref{fattiduale}(a). Now apply the previous implication.

(b) It is clear that $U/\bigl(U\cap \phi(U)\bigr)\cong\bigl(U+\phi (U)\bigr)/U$ and so
\begin{equation}\label{scaladuale1}
s(\phi,U) = \Bigl\lvert \frac{U}{U\cap\phi(U)}\Bigr\rvert =\Bigl\lvert \frac{U+\phi(U)}{U}\Bigr\rvert.
\end{equation}
Moreover, let $F= \bigl(U+\phi(U)\bigr)/U$; then $F$ is a finite abelian group and so $F\cong\widehat F$. Therefore, by \eqref{scaladuale1} and Fact \ref{fattiduale}, we have that
\begin{equation}\label{scaladuale2}
s(\phi,U) = \lvert F\rvert =\lvert \widehat{F}\rvert=
\Bigl\lvert\frac{U^\bot}{(U+\phi (U))^\bot}\Bigr\rvert=\Bigl\lvert \frac{U^\bot}{U^\bot \cap 
\widehat{\phi}^{-1}( U^\bot)} \Bigr\rvert.
\end{equation}
Finally, since $\widehat{\phi}$ is an automorphism,
\begin{equation}\label{scaladuale3}
\Bigl\lvert \frac{U^\bot}{U^\bot \cap \widehat{\phi}^{-1}( U^\bot)} \Bigr\rvert=\Bigl\lvert \frac{\widehat{\phi}(U^\bot)}{\widehat{\phi}(U^\bot )\cap U^\bot} \Bigr\rvert=s(\widehat{\phi},U^\bot).
\end{equation}
Now \eqref{scaladuale2} and \eqref{scaladuale3} give immediately the equality in (b).

(c) Let $U\in\mathcal M(G,\phi)$. If $V \in \mathcal B(\widehat{G})$, then $V^\bot  \in \mathcal B({G})$ by item (a), and $s(\phi,V^\bot) = s(\widehat{\phi},V)$ by item (b). So, applying twice (b), we have
\begin{equation*}\label{equazionefinale}
s(\widehat{\phi},U^\bot) = s(\phi,U) \leq  s(\phi,V^\bot) = s(\widehat{\phi},V).
\end{equation*}
Since this inequality holds true for all $V \in \mathcal B(\widehat{G})$, we can conclude that $U^\bot \in\mathcal M(\widehat G,\widehat{\phi})$.

To prove the converse implication apply Fact \ref{fattiduale}(a) and the previous implication.

(d) follows from (c).
\end{proof}

\end{document}